%% file: main.tex
\documentclass[10pt]{amsart}

\makeatletter
\def\@settitle{\begin{center}%
  \baselineskip14\p@\relax
    \normalfont\LARGE%<- NEW

  \@title
  \end{center}%
}
\makeatother

\usepackage[dvipsnames]{xcolor}
\usepackage[hyperfootnotes=false]{hyperref}
\usepackage{leftindex}
\usepackage{enumerate}
\usepackage{tikz}
\usepackage{tikz-cd}
\usepackage{hyperref}
\usepackage[capitalize, noabbrev, nosort]{cleveref}
\usepackage[bbgreekl]{mathbbol}
\usepackage{amsthm,amsfonts,amsmath,amscd,amssymb}
\usepackage{xcolor,float}
\usepackage[all]{xy}
\usepackage{mathrsfs}
\usepackage{graphics,graphicx}
\usepackage{caption}
\usepackage{comment}
\usepackage{array}
\usepackage{enumitem}
\usepackage{stmaryrd}
\usepackage{tabularx}
\usepackage{xspace}

\newcolumntype{P}[1]{>{\centering\arraybackslash}p{#1}}
\newcolumntype{M}[1]{>{\centering\arraybackslash}m{#1}}

\DeclareMathAlphabet\EuScript{U}{eus}{m}{n}
\SetMathAlphabet\EuScript{bold}{U}{eus}{b}{n}

\usepackage{epigraph}
\usepackage{bm}

\let\oldmarginpar\marginpar
\renewcommand\marginpar[1]{\-\oldmarginpar[\raggedleft\footnotesize #1]%
	{\raggedright\footnotesize #1}}
\theoremstyle{plain}

\newtheorem{thm}{Theorem}[section]
\newtheorem{lemma}[thm]{Lemma}
\newtheorem*{theorem*}{Theorem}
\newtheorem*{corollary*}{Corollary}

\theoremstyle{definition}

\newtheorem{example}[thm]{Example}

\newtheorem{remark}[thm]{Remark}

\numberwithin{equation}{section}

\usepackage{mdframed}
\usepackage{lipsum}
\newtheorem*{theo}{Comparison Lemma}
\newenvironment{ftheo}
  {\begin{mdframed}[innertopmargin = 3pt, innerbottommargin=3pt,skipabove=5pt,skipbelow=5pt,linewidth=0.25pt,nobreak=true,align=center]\begin{theo}}
  {\end{theo}\end{mdframed}}
\def\maintheoref{the \hyperref[thm:main_result_equality_DO_NOT_CREF]{Comparison Lemma}\xspace}

\newcommand{\F}{\mathsf{F}}

\newcommand{\N}{\mathbb{N}}
\newcommand{\Z}{\mathbb{Z}}
\newcommand{\Q}{\mathbb{Q}}
\newcommand{\R}{\mathbb{R}}

\newcommand{\SL}{\on{SL}}

\newcommand{\La}{\Lambda}

\renewcommand{\b}{\beta}

\newcommand{\dd}{\partial}

\newcommand{\sse}{\subset}
\newcommand{\lr}{\longrightarrow}

\newcommand{\Br}{\operatorname{Br}}

\newcommand{\st}{\text{st}}

\newcounter{daggerfootnote}

\newcommand{\G}{\mathsf{G}}
\newcommand{\T}{\mathrm{T}}
\newcommand{\U}{\mathsf{U}}

\newcommand{\s}{\sigma}

\newcommand{\sw}{\mathsf{sw}}
\newcommand{\nw}{\mathsf{nw}}
\newcommand{\se}{\mathsf{se}}

\def\on{\operatorname}

\usepackage{dsfont}
\allowdisplaybreaks

\def \vertbar [#1](#2,#3,#4){
    \draw [#1] (#2,#3) -- (#2,#4);
    \draw [fill=white] (#2,#3) circle [radius=0.1];
    \draw [fill=black] (#2,#4) circle [radius=0.1];
}

\usepackage{mathdots}
\usepackage{afterpage}
\makeatletter
\providecommand{\leftsquigarrow}{%
  \mathrel{\mathpalette\reflect@squig\relax}%
}
\newcommand{\reflect@squig}[2]{%
  \reflectbox{$\m@th#1\rightsquigarrow$}%
}
\makeatother

\makeatletter
\def\Ddots{\mathinner{\mkern1mu\raise\p@
\vbox{\kern7\p@\hbox{.}}\mkern2mu
\raise4\p@\hbox{.}\mkern2mu\raise7\p@\hbox{.}\mkern1mu}}
\makeatother

\def \horline [#1](#2,#3,#4){
    \draw [#1] (#2,#4) -- (#3,#4);
    \draw [fill=white] (#2,#4) circle [radius=0.1];
    \draw [fill=black] (#3,#4) circle [radius=0.1];
}

\def \crossing (#1,#2)(#3,#4){
\draw (#1,#2) -- (#3,#4);
\draw (#1,#4) -- (#3,#2);
}

\DeclareFontFamily{U}{mathb}{}
\DeclareFontShape{U}{mathb}{m}{n}{
  <-5.5> mathb5
  <5.5-6.5> mathb6
  <6.5-7.5> mathb7
  <7.5-8.5> mathb8
  <8.5-9.5> mathb9
  <9.5-11.5> mathb10
  <11.5-> mathbb12
}{}

\usetikzlibrary{decorations.markings, arrows}
\tikzset{tangent/.style={decoration={markings,mark=at position #1 with {
      \coordinate (tangent point-\pgfkeysvalueof{/pgf/decoration/mark info/sequence number}) at (0pt,0pt);
      \coordinate (tangent unit vector-\pgfkeysvalueof{/pgf/decoration/mark info/sequence number}) at (1,0pt);
      \coordinate (tangent orthogonal unit vector-\pgfkeysvalueof{/pgf/decoration/mark info/sequence number}) at (0pt,1);
      }},postaction=decorate},
    use tangent/.style={
        shift=(tangent point-#1),
        x=(tangent unit vector-#1),
        y=(tangent orthogonal unit vector-#1)
    },
    use tangent/.default=1
    }

\begin{document}

	\title{On wall-crossing coordinates in Cerf theory}

\author{Roger Casals}
	\address{University of California Davis, Dept. of Mathematics, USA}
	\email{casals@ucdavis.edu}

	\subjclass[2010]{Primary: 53D10. Secondary: 57K43, 13F60.}
		
\maketitle
%\vspace{-1.2cm}
\begin{abstract} We relate Bruhat numbers in real Morse theory to cluster variables in braid varieties. This provides instances of wall-crossing coordinates in the study of Cerf diagrams.
\end{abstract}

%\tableofcontents

%%%%%%%%%%%%%%%%%%%%%%%%%%%%%%%%%%%%%%%%%%%%%%%%%%%%%%%%%%%%%%%%%
%%%%%%%%%%%%%%%%%%%%%%%%%%%%%%%%%%%%%%%%%%%%%%%%%%%%%%%%%%%%%%%%%
%%%%%%%%%%%%%%%%%%%%%%%%%%%%%%%%%%%%%%%%%%%%%%%%%%%%%%%%%%%%%%%%%

\input{0_introduction}

%%%%%%%%%%%%%%%%%%%%%%%%%%%%%%%%%%%%%%%%%%%%%%%%%%%%%%%%%%%%%%%%%
%%%%%%%%%%%%%%%%%%%%%%%%%%%%%%%%%%%%%%%%%%%%%%%%%%%%%%%%%%%%%%%%%
%%%%%%%%%%%%%%%%%%%%%%%%%%%%%%%%%%%%%%%%%%%%%%%%%%%%%%%%%%%%%%%%%

\input{1_Coordinates}

%%%%%%%%%%%%%%%%%%%%%%%%%%%%%%%%%%%%%%%%%%%%%%%%%%%%%%%%%%%%%%%%%
%%%%%%%%%%%%%%%%%%%%%%%%%%%%%%%%%%%%%%%%%%%%%%%%%%%%%%%%%%%%%%%%%
%%%%%%%%%%%%%%%%%%%%%%%%%%%%%%%%%%%%%%%%%%%%%%%%%%%%%%%%%%%%%%%%%

\input{2_Examples}

%%%%%%%%%%%%%%%%%%%%%%%%%%%%%%%%%%%%%%%%%%%%%%%%%%%%%%%%%%%%%%%%%
%%%%%%%%%%%%%%%%%%%%%%%%%%%%%%%%%%%%%%%%%%%%%%%%%%%%%%%%%%%%%%%%%
%%%%%%%%%%%%%%%%%%%%%%%%%%%%%%%%%%%%%%%%%%%%%%%%%%%%%%%%%%%%%%%%%

%\input{3_Questions}

%%%%%%%%%%%%%%%%%%%%%%%%%%%%%%%%%%%%%%%%%%%%%%%%%%%%%%%%%%%%%%%%%
%%%%%%%%%%%%%%%%%%%%%%%%%%%%%%%%%%%%%%%%%%%%%%%%%%%%%%%%%%%%%%%%%
%%%%%%%%%%%%%%%%%%%%%%%%%%%%%%%%%%%%%%%%%%%%%%%%%%%%%%%%%%%%%%%%%

\bibliographystyle{alpha}
\bibliography{main}

\end{document}

%% file: 0_introduction.tex
\section{Introduction}\label{sec:intro}

The object of this short note is to relate Bruhat numbers in real Morse theory to cluster variables in braid varieties. Specifically, the Bruhat numbers introduced in \cite{pushkar2021enhancedbruhatdecompositionmorse}, refining \cite{barannikov:hal-01745109}, and the cluster structures constructed in \cite{CGGLSS}, building on \cite{CasalsGao24,CasalsWeng,CasalsZaslow}. A motivating factor is J.~Cerf's seminal work \cite{Cerf70} in parametric Morse theory.

\begin{center}
	\begin{figure}[h!]
		\centering
		\includegraphics[width=\textwidth]{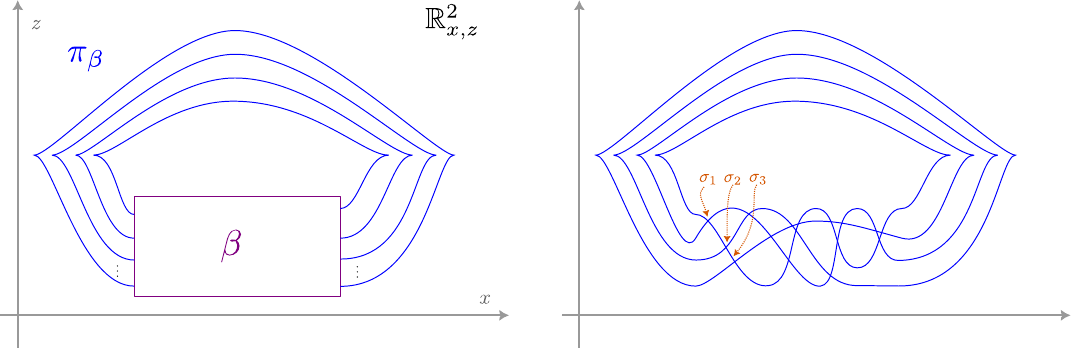}
		\caption{(Left) The general form of a Cerf diagram $\pi_\beta$ associated to a positive braid word $\beta$. Each positive crossing of the braid word corresponds to a value exchange between critical points. (Right) An example with $\beta=(\s_1\s_2\s_3)^2
    \s_1\s_2\s_1\s_3(\s_1\s_2\s_1)^2$.}\label{fig:FrontsBeta}
	\end{figure}
\end{center}

\subsection{Scientific context} A starting point is the Cerf diagram $\pi_\beta$ associated to a positive braid word $\beta$, as depicted in \cref{fig:FrontsBeta} (left). The $z$-axis records the values of the critical points for a 1-parametric family of generalized Morse functions $f_x:\R^n\lr\R$, $x\in\R$, where $f_{x}$ has no critical points if $|x|$ is large enough, and all $f_x$ are assumed linear at infinity. This class of Cerf diagrams lies at the core of the global study of 1 and 2-parametric Morse theory, cf.~e.g.~\cite[Th\'eor\`eme 2']{Cerf70} and its proof.\footnote{In \cite{Cerf70}, the study of the homotopy type of the corresponding spaces of (generalized) Morse functions, focused on connectivity, has three parts: local, semi-local and global. See also the outline \cite{Cerf_outline}.} In fact, the semi-local lemmas in \cite[Chapitres II-IV]{Cerf70} are used to reduce any Cerf diagram with the same boundary conditions, on $x$ and each $f_x$, to one of such form. In the {\it lemme fondamental} of \cite[Chapitre VI]{Cerf70}, generators of a key relative fundamental group are described by using the Bruhat decomposition of $\on{GL}_n$, indexed by its Weyl group $S_n$. For instance, the generator denoted $\delta_{i,g}$ in ibid is described by a 1-parametric family of Morse functions realizing a value exchange between points of Morse index $i+1$. The use of the Bruhat decomposition in the context of (framed) Morse functions is further elucidated by S.~Barannikov in \cite{barannikov:hal-01745109}, see also the enlightening results of \cite{laudenbach2015articlesabarannikov,peutrec2011precisearrheniuslawpforms}, and more recently by P.~Pushkar and M.~Tyomkin in \cite{pushkar2021enhancedbruhatdecompositionmorse}.

Independently, cluster transformations provide a framework for the study of wall-crossing phenomena in algebraic and symplectic geometry, cf.~e.g.~\cite{GHKK,KontsevichSoibelman_Wallcrossing,GMN_Cluster}.\footnote{See also the recorded lectures at the SLMath workshop ``\href{https://www.slmath.org/workshops/783\#overview_workshop}{Cluster algebras and wall-crossing}'' in March 2016.} Specifically, the cluster variables of a cluster algebra determine exponential Darboux coordinates for the chambers, i.e.~the cluster charts, and their mutations are related to the walls of the associated scattering diagram, cf.~\cite[Section 1]{GHKK}. One of the first significant instances of cluster algebras were coordinate rings of Bruhat cells and variations thereof, cf.~\cite{BFZ05}, where the chamber ansatz dictates the cluster variables in terms of generalized minors. The Bruhat decomposition, both in its Borel and unipotent forms, has a crucial role in the study of such cluster algebras. The cluster structures on braid varieties, constructed in \cite{CGGLSS} via the weave calculus of \cite{CasalsWeng,CasalsZaslow}, can also be understood in terms of the Bruhat decomposition and generalized minors, as explained in \cite{casals2025comparingclusteralgebrasbraid}.

\noindent In short, a reading of \cite{Cerf70} and the crucial appearance of the Bruhat decomposition in cluster algebras lead the author to expect that J.~Cerf's results in \cite{Cerf70} adumbrate the presence of wall-and-chamber structures in the study of real parametric Morse theory. The Comparison Lemma below is an exercise aimed at illustrating this insight.

\subsection{Main result} This note will discuss the following lemma, comparing \textcolor{purple}{cluster coordinates}, in red, to \textcolor{blue}{Bruhat numbers}, in blue:

\begin{ftheo}\label{mainlemma}
Let $\beta=\s_{i_1}\s_{i_2}\cdots \s_{i_l}$ be a positive braid word, $(\pi_\beta,\{f_x\})$ its associated Cerf diagram, and $X(w_0\b)\sse\on{Spec}\Z[z_1,\ldots,z_l]$ its braid variety. Then the cluster variable $A_k\in\Z[X(w_0\b)]$ equals
\begin{equation}\label{eq:mainidentity}
\textcolor{purple}{A_k(z_1,\ldots,z_k)}=\textcolor{blue}{\on{sgn}_k\cdot \prod_{j=1}^{i^*_k} \bbbeta_j(f_{x_k})}
\end{equation}
where $x_k\in\R$ is arbitrarily close but greater than the $x$-coordinate of the $k$th crossing of $\pi_\beta$, $\bbbeta_j(f_{x_k})$ is the $j$th Bruhat number of $f_{x_k}$, and $\on{sgn}_k:=(-1)^{\lfloor i_k^*/2\rfloor+i_k^*(n-1)}$ with $i_k^*:=n-i_k$.
\end{ftheo}

The Comparison Lemma is proven in \cref{ssec:proof_comparisonlemma}. An intuitive description of it is as follows. Let $\beta=\s_{i_1}\s_{i_2}\cdots \s_{i_l}$ be a positive braid word in $n$-strands and of length $l\in\N$. There are two types of data associated to $\beta$ that have been studied in the literature:

\begin{enumerate}
    \item {\it \textcolor{Blue}{Morse data}}: the Cerf diagram $\pi_\beta$, as depicted in \cref{fig:FrontsBeta}. This is data about the critical points of 1-parametric families of real Morse functions. These are studied in depth in \cite{Cerf70}.\\
    
    \item {\it \textcolor{purple}{Cluster data}}: the cluster algebra $\Z[X(w_0\beta)]$, studied in \cite{CGGLSS}. The relation between cluster algebras and wall-crossing is studied in \cite{GHKK,KontsevichSoibelman_Wallcrossing}.
\end{enumerate}

For the Morse data, we fix an index $\iota\in\N$ large enough and assume that at every cusp of the Cerf diagram the Morse index at the upper strand always equals $(\iota+1)$. Thus the Morse index at the lower strand of all cusps is $\iota$. We also assume that all handleslides shall occur precisely to the left of a value exchange, i.e.~a crossing, or a death, i.e.~a right cusp. This is the $A$-form studied in \cite{Henry_2015}, cf.~also \cite[Section 5]{HenryRutherford15}. Note that any Cerf diagram decorated with handleslides can be connected, via Cerf diagrams, to a Cerf diagram in normal form: this follows from the moves \cite[Figure 3.6-3.8]{Henry09_Thesis} refined integrally and at the function level. The handleslide mark for the handleslide associated to the $k$th crossing will be denoted by the variable $z_k$, $k\in[1,l]$. We denote by $(\pi_\beta,\{f_x\})$ any pair consisting of a 1-parametric family of Morse functions $\{f_x\}$ in this normal form, $x\in\R$, with $\pi_\beta$ as its Cerf diagram.\\

   \noindent In the $(x,z)$-coordinates of \cref{fig:FrontsBeta}, we always assume that the 1-parametric family of generalized Morse functions $f_x:\R^N\lr\R$ is on some $\R^N$, for $N\in\N$ fixed and large enough, and linear at infinity. The Bruhat numbers of a Morse function $f_x$ are introduced in \cite{pushkar2021enhancedbruhatdecompositionmorse} and will be denoted by $\bbbeta_j(f_x)$, $j\in[1,n]$.\\

For the cluster data, cluster structures on braid varieties are studied in \cite{CGGLSS}, cf.~also \cite{CGGS,CasalsWeng}: braid varieties are a class of smooth algebraic varieties over $\Z$ and their $\Z$-algebras of regular functions $\Z[X(w_0\beta)]$ are shown to be cluster algebras in \cite[Theorem 1.1]{CGGLSS}. By definition, $X(w_0\beta)$ is a moduli space parametrizing linear tuples of flags whose relative positions are dictated by $\beta$. There is a defining embedding $X(w_0\beta)\sse\Z^l$, where the ambient coordinates $z_1,\ldots,z_l\in\Z^l$ can be chosen to be in bijection with the crossings of $\beta$, $z_k$ associated to $\s_{i_k}$. We denote the cluster variables of the initial seed given by the right-inductive weave of $\beta w_0$ by $A_k$, $k\in[1,l]$, cf.~\cite[Section 5.3]{CGGLSS}. By construction, the cluster variable $A_k=A_k(z_1,\ldots,z_k)\in\Z[X(w_0\beta)]$ can be expressed as a regular function on the first $k$ ambient variables $z_1,\ldots,z_k$. In the context of wall-crossing, the chamber corresponding to this initial cluster seed is given by the open set $\{A_1\neq0,\ldots,A_l\neq0\}\sse X(w_0\beta)$.

In a nutshell, to connect the Morse data and the cluster data, the flags parametrized by a point in $X(w_0\beta)$ can be understood as a filtration of the integral homologies of the sublevel sets of the functions $f_x$ producing the Cerf diagram $\pi_\beta$. Intuitively, $X(w_0\beta)$ is a finite-dimensional model that parametrizes part of the space of all 1-parametric families of Morse functions with Cerf diagram $\pi_\beta$, cf.~\cref{ssec:cats} for more context. The cluster variable $A_k$ on $X(w_0\beta)$, on the left hand side of \maintheoref, algebraically measures the transversality between the initial standard flag and the flag right after the $k$th crossing. In line with this, and appropriately described, the products of certain Bruhat numbers of $f_{x}$ can be understood as a measure of the relative position of $f_x$ with respect to $f_{-\infty}$: such relative position can be expressed in terms of the handleslides and value exchanges needed to go from $f_{-\infty}$ to the given $f_x$, and this is quantitatively measured by the right hand side of \maintheoref.\\

%%%%%%%%%%%%%%%%%%%%%%%%%%%%%%%%%%%%%%%%%%%%%%%%%%%%%%%%%%%%%%%%%%%%%%%%
%%%%%%%%%%%%%%%%%%%%%%%%%%%%%%%%%%%%%%%%%%%%%%%%%%%%%%%%%%%%%%%%%%%%%%%%
%%%%%%%%%%%%%%%%%%%%%%%%%%%%%%%%%%%%%%%%%%%%%%%%%%%%%%%%%%%%%%%%%%%%%%%%

\noindent{\bf Acknowledgements}: I am grateful to the organizers of the 2025 Georgia International Topology Conference who, once more, created a wonderful environment for geometers and topologists from all around to gather, interact, and discuss exciting new developments in this area of mathematics. I also thank F.~Laudenbach, whose work on Morse theory continues to inspire. The author is supported by the National Science Foundation via DMS-2505760 and DMS-1942363.\hfill$\Box$

%%%%%%%%%%%%%%%%%%%%%%%%%%%%%%%%%%%%%%%%%%%%%%%%%%%%%%%%%%%%%%%%%%%%%%%%
%%%%%%%%%%%%%%%%%%%%%%%%%%%%%%%%%%%%%%%%%%%%%%%%%%%%%%%%%%%%%%%%%%%%%%%%
%%%%%%%%%%%%%%%%%%%%%%%%%%%%%%%%%%%%%%%%%%%%%%%%%%%%%%%%%%%%%%%%%%%%%%%%

%\noindent{\bf Notation}: \hfill$\Box$

%% file: 1_Coordinates.tex
\section{The argument}

The goal of this section is to prove \maintheoref, which is done in \cref{ssec:proof_comparisonlemma}. The reader is referred to the excellent article \cite{pushkar2021enhancedbruhatdecompositionmorse} for the enhanced Bruhat decomposition and the Bruhat numbers of Morse functions, and to \cite{CGGLSS} for braid varieties and cluster structures on their rings of functions.

\subsection{Initial setup} Let $\beta=\s_{i_1}\s_{i_2}\cdots \s_{i_l}\in\on{Br}^+_n$ be an $n$-stranded positive braid word and $(\pi_\beta,\{f_x\})$ its associated Cerf diagram, as depicted in \cref{fig:FrontsBeta}, with $f_x:\R^N\lr\R$ a 1-parametric family of generalized Morse functions, in normal $A$-form and linear at infinity. Let $f:\R^n\lr\R$ be a Morse function of the form $f=f_x$ where $x$ lies in-between the $x$-coordinates of the two sets of nested cusps of $\pi_\beta$, and it is not the $x$-coordinate of any crossing either. Such an $f$ has only critical points of indices $\iota$ and $\iota+1$, and precisely $n$ of each. The critical points of index $\iota$ are denoted by $q_1,\ldots,q_n$ and those of index $\iota+1$ by $p_1,\ldots,p_n$, and we always assume
$$f(q_n)<f(q_{n-1})<\ldots<f(q_1)<f(p_{n-1})<f(p_{n-2})<\ldots<f(p_1).$$

\noindent For the corresponding elements of the Morse complex $CM_*(f)$, $q_i\in CM_{\iota}(f)$ and $p_i\in CM_{\iota+1}(f)$ are identified with the column vector $e_i=(0,\ldots,0,1,0,\ldots,0)^t$ with the unique entry 1 in the $i$th position. This explicitly describes isomorphisms $CM_{\iota} (f)\cong\Z^n$ and $CM_{\iota+1}(f)\cong\Z^n$.

\begin{example}\label{ex:leftmost_differential} Consider an $x$-coordinate $x_0$ with value arbitrarily close but greater than the $x$-coordinate of the rightmost left cusp of $\pi_\beta$. In these coordinates above, and given that no handleslides occur before such an $x_0$ value, the Morse differential $\dd$ of $f_x$ can be expressed as
$$\dd:\Z^n\lr\Z^n,\quad \dd(p_i)=(-1)^{i+1} q_{w_0(i)},\quad i\in[1,n],$$
where $w_0$ is the longest element of $S_n$, i.e.~the half-twist. In particular, this differential $\dd$ can be expressed as a permutation matrix in $\SL_n$ lifting the element $w_0$ from its Weyl group $S_n$.
\qed
\end{example}

\subsection{An explicit pinning}

For $\SL_n$, we choose the following matrices $x_i(z)$ and $B_i(z)$:
\begin{equation}\label{eq:pinning1}
x_{i}(z) := \begin{pmatrix}
1 & \cdots  & & & \cdots & 0\\
\vdots & \ddots & & & & \vdots\\
0 & \cdots & 1 & z & \cdots & 0\\
0 & \cdots & 0 & 1 & \cdots & 0\\
\vdots &  & & &\ddots & \vdots\\
0 & \cdots & & & \cdots & 1\\
\end{pmatrix},\qquad
B_{i}(z) :=\begin{pmatrix}
1 & \cdots  & & & \cdots & 0\\
\vdots & \ddots & & & & \vdots\\
0 & \cdots & z & -1 & \cdots & 0\\
0 & \cdots & 1 & 0 & \cdots & 0\\
\vdots &  & & &\ddots & \vdots\\
0 & \cdots & & & \cdots & 1\\
\end{pmatrix},
\end{equation}
where the $(2\times 2)$-submatrices sit at the $i$th and $(i + 1)$st rows and columns. The variable $z$ is thus the $(i,i+1)$-entry of $x_i(z)$, and the $(i,i)$-entry of $B_i(z)$. This can be seen as part of a choice of pinning for $\SL_n$, cf.~e.g.~ \cite[Section 3.4]{CGGLSS}. Similarly, consider the $P_i(z)$ matrix
\begin{equation}\label{eq:pinning1}
P_{i}(z) := \begin{pmatrix}
1 & \cdots  & & & \cdots & 0\\
\vdots & \ddots & & & & \vdots\\
0 & \cdots & 0 & 1 & \cdots & 0\\
0 & \cdots & -1 & z & \cdots & 0\\
\vdots &  & & &\ddots & \vdots\\
0 & \cdots & & & \cdots & 1\\
\end{pmatrix},
\end{equation}
where the $(2\times 2)$-submatrices sit at the $i$th and $(i + 1)$st rows and columns. Given a positive braid word $\beta=\s_{i_1}\cdots\s_{i_l}$ we define
$$B_\beta(z_1,\ldots,z_l):=B_{i_1}(z_1)\cdots B_{i_l}(z_l)\qquad\mbox{and}\qquad P_\beta(z_1,\ldots,z_l):=P_{i_l}(z_l)\cdots P_{i_1}(z_1).$$
Note that $P_i(z)^{-1}:=B_i(z)$ and thus
\begin{equation}\label{eq:B_is_inverseP}
P_{\beta}(z)^{-1}=B_\beta(z).
\end{equation} 
%where $\overleftarrow{\beta}=\s_{i_l}\cdots\s_{i_1}$.

\noindent We denote by $\Delta^\sw_i$, $\Delta^\se_i$ and $\Delta^\nw_i$, the $i$th principal minor starting from the lower-left corner (south-west), the lower-right corner (south-east) and the upper-left corner (north-west), respectively. In particular, $\Delta^\nw_i$ is the $i$th leading principal minor, i.e.~the generalized minor for the fundametnal weight $\omega_i$ in $\SL_n$, cf.~e.g.~\cite[Section 2.3]{BFZ05}.

\begin{example}\label{ex:minors} Let ${\U\sse\SL_n}$ be the unipotent subgroup of upper unitriangular matrices and $\T\sse\SL_n$ the Cartan subgroup of diagonal matrices. Suppose that $M\in\SL_n$ lies in the unipotent Bruhat cell $M\in\U w_0\T\U\sse\SL_n$, then
$$\Delta^\sw_i(M)=(-1)^{i(i-1)/2}\cdot\Delta^\nw_i(t),$$
where $t\in T$ is the Cartan representative of $M$ in the cell. Indeed, the Cauchy-Binet formula implies that $\Delta^\sw_i(M)=\Delta^\sw_i(w_0t)$, while $\Delta^\sw_i(w_0t)=(-1)^{i(i-1)/2}\Delta^\nw_i(t)$ is a direct computation.\qed

\end{example}

\subsection{Coordinatizing Morse Differentials}\label{ssec:coordinatizing_Morse} The Morse differentials $\dd^h_{l}$ and $\dd^h_r$ before and after a handleslide of $q_i$ along $q_{i-1}$ with handleslide marked by a variable $z$, $i\in[2,n]$, are related by
\begin{equation}\label{eq:differentials_handleslides}
\dd^{(h)}_r=x_i(-z)\dd^{(h)}_l.
\end{equation}
See \cref{fig:Handleslide_Crossings} (left) for a depiction of the location of such differentials in part of the Cerf diagram.
\begin{center}
	\begin{figure}[h!]
		\centering
		\includegraphics[scale=1]{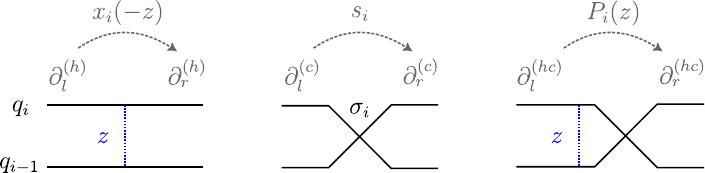}
		\caption{A depiction of \cref{eq:differentials_handleslides} on the left, \cref{eq:differentials_crossing} at the center, and \cref{eq:differentials_handleslide_crossing} on the right. The Cerf diagram is in black, the handleslide marks in dashed blue, and the location and relation between the Morse differentials in gray.}\label{fig:Handleslide_Crossings}
	\end{figure}
\end{center}

The Morse differentials $\dd^c_{l}$ and $\dd^c_r$ before and after exchanging the values of $q_{i-1}$ and $q_i$, $i\in[2,n]$, i.e.~a $\s_i$ crossing in the Cerf diagram as in \cref{fig:Handleslide_Crossings} (center), are related by
\begin{equation}\label{eq:differentials_crossing}
\dd^{(c)}_r=s_i\dd^{(c)}_l,
\end{equation}
where $s_i=P_i(0)$ is a permutation matrix lifting the simple transposition $s_i\in S_n$.

\begin{remark} \noindent $(i)$ The Bruhat numbers are defined independently of a choice of metric, cf.~\cite[Section 0.2]{pushkar2021enhancedbruhatdecompositionmorse} and thus cannot change under a handleslide. This is reflected algebraically in \cref{eq:differentials_handleslides} above: since $x_i(z)$ is unipotent upper-triangular, the rook matrices associated to any matrix $M$ and $x_i(\pm z)M$ must coincide, and thus their Bruhat numbers are equal. That said, the Bruhat numbers after a handleslide between two critical points followed by an exchange of their values do depend on the handleslide mark.

\noindent $(ii)$ In Example \ref{ex:leftmost_differential}, Equations (\ref{eq:differentials_handleslides}) and (\ref{eq:differentials_crossing}) the minus signs in $\dd$, $x_i(-z)$ and $s_i$ are due to suitable choices of relative orientations of the stable and unstable cells, made to match the pinning for $\SL_n$ in \cite[Section 3.4]{CGGLSS}.\qed
\end{remark}

\noindent By Equations (\ref{eq:differentials_handleslides}) and (\ref{eq:differentials_crossing}), the differentials $\dd^{(hc)}_{l}$ and $\dd^{(hc)}_r$ before and after a pair of handleslide mark $z$ and a value exchange, both between $q_i$ and $q_{i-1}$ as in \cref{fig:Handleslide_Crossings} (right), $i\in[2,n]$, satisfy
\begin{equation}\label{eq:differentials_handleslide_crossing}
\dd^{(hc)}_r=s_i\dd^{(c)}_l=s_ix_i(-z)\dd^{(hc)}_l=P_i(z)\dd^{(hc)}_l,
\end{equation}
as $s_i\cdot x_i(-z)=P_i(z)$ by direct computation. Therefore, in a Cerf diagram where all the handleslides occur exactly to the left of a crossing, the differential $\dd_k$ right after the $k$th crossing is given by
\begin{equation}\label{eq:differentials_crossing2}
\dd_k=P_{i_k}(z_k)\cdots P_{i_2}(z_2)P_{i_1}(z_1)\dd_0,
\end{equation}
where $\dd_0=P_{w_0}(0)$ is a permutation matrix for the longest permutation $w_0\in S_n$, as in Example \ref{ex:leftmost_differential}.

\begin{lemma}[Bruhat numbers from handleslide marks]\label{prop:bruhatnumbers} Let $\beta=\s_{i_1}\s_{i_2}\cdots\s_{i_l}$ and $(\pi_\beta,\{f_x\})$ its associated Cerf diagram in normal form. Consider the Morse differential $\dd_k$ of a Morse function $f_k$ right after the $k$th crossing, and its Bruhat numbers $\bbbeta_j(f_k)$, $j\in[1,n]$. Then
\begin{equation}\label{eq:bruhatnumbers}
\displaystyle\prod_{j=1}^{i^*_k} \bbbeta_j(f_k) =(-1)^{i^*_k(i^*_k-1)/2}\cdot\Delta^{\sw}_{i^*_k}(P_{i_k}(z_{i_k})\cdots P_{i_1}(z_1)w_0).
\end{equation}
if $P_{i_k}(z_{i_k})\cdots P_{i_1}(z_1)w_0\in \U w_0\T\U$ lies in the unipotent Bruhat cell for $w_0\in S_n$.
\end{lemma}

\begin{proof}
The statement follows from \cref{eq:differentials_crossing2} and Example \ref{ex:minors}.
\end{proof}

\subsection{Proof of \maintheoref}\label{ssec:proof_comparisonlemma} \cref{eq:mainidentity} is obtained from the following equalities:
\begin{equation*}
\begin{split}
\displaystyle\prod_{j=1}^{i^*_k} \bbbeta_j(f_k) & \stackrel{(1)}{=}\rho_1(k)\cdot\Delta^{\sw}_{i^*_k}(P_{i_k}(z_{i_k})\cdots P_{i_1}(z_1)w_0) \\
& \stackrel{(2)}{=} \rho_1(k)\rho_2(k)\cdot\Delta^{\se}_{i_k^*}( P_{i_k}(z_{i_k})\cdots P_{i_1}(z_1)) \\
& \stackrel{(3)}{=} \rho_1(k)\rho_2(k)\cdot \Delta^{\nw}_{i_k}( B_{i_1}(z_{1})\cdots B_{i_k}(z_k)) \\
& \stackrel{(4)}{=} \rho_1(k)\rho_2(k)\cdot A_k(z_1,\ldots,z_k).
\end{split}
\end{equation*}
Here we denoted $\rho_1(k):=(-1)^{i_k^*(i^*_k-1)/2}=(-1)^{\lfloor i^*_k/2\rfloor}$ and $\rho_2(k):=(-1)^{i_k^*(n-1)}$ for the signs. Each of the equalities can be justified as follows:

\begin{enumerate}
    \item Equality 
(1) is \cref{eq:bruhatnumbers} in Lemma \ref{prop:bruhatnumbers}. Since $P_{i_k}(z_{i_k})\cdots P_{i_1}(z_1)w_0\in\U w_0\T\U$ is an open condition, we can and do assume it here: the resulting equalities under this assumption extend to the equalities between the corresponding global regular functions on $X(w_0\beta)$.\\

\item Equality 
(2) is implied by the linear algebra fact that $\Delta^{\sw}_{m}(Mw_0)=(-1)^{m
(n-1)}\cdot\Delta^{\se}_{m}(M)$ for any $M\in\on{Mat}_{n\times n}(\Z)$, applied to $m=i_k^*$ and $M=P_{i_k}(z_{i_k})\cdots P_{i_1}(z_1)$.\\

\item Equality 
(3) follows from the fact that the $m$th leading principal minor of an invertible $(n\times n)$-matrix coincides with the $(n-m)$-th trailing principal minor of its inverse. That is, we are using the linear algebra identity $\Delta^{\se}_{m}(M)=\Delta^{\nw}_{n-m}(M^{-1})$ for any $M\in\on{GL}_{n}(\Z)$, applied to $m=i_k^*$ and $M=P_{i_k}(z_{k})\cdots P_{i_1}(z_1)$, and so $M^{-1}=B_{i_1}(z_{1})\cdots B_{i_k}(z_k)$ by \cref{eq:B_is_inverseP}.\\

\item Equality 
(4) is a consequence of \cite[Theorem 5.12]{CGGLSS} and the last line of equalities in \cite[Section 3.7]{CGGLSS}, where the cluster variables are expressed as generalized minors.\qed
\end{enumerate}

\begin{remark}
$(i)$ Note that $\bbbeta_j(f_k)\in\Q(z_1,\ldots,z_2)$ are rational functions on the $z$-variables, and typically not polynomial. In fact, their definition requires a field. Nevertheless, \maintheoref and \cite[Theorem 7.6]{CGGLSS} together imply that the product of Bruhat numbers on the right hand side of \cref{eq:mainidentity} is a polynomial in the $z$-variables with integer coefficients.

\noindent $(ii)$ The equations cutting out a braid variety directly translate in Cerf theory as the conditions for the critical points of $f_x$ to be able to cancel in pairs as dictated by the right cusps when $x$ is close but smaller than the $x$-coordinate of the leftmost right cusp. Such conditions themselves cut out equations in terms of the handleslide marks, and those coincide with the equations in \cite[Corollary 3.7]{CGGLSS}, cf.~also \cite[Definition 3.15 \& Lemma 3.16]{CGGLSS}.\qed
\end{remark}

\subsection{A final comment}\label{ssec:cats} There is a precise connection between parametric Morse theory and the study of Legendrian submanifolds. For instance, normal rulings, as introduced and studied by D.~Fuchs, P.~Pushkar and Y.~Chekanov and their collaborators, combinatorially capture Barannikov's pairing \cite{barannikov:hal-01745109} of critical points. Another instance is the theory of generating families for Legendrian submanifolds and their associated invariants, as developed by many, including C.~Viterbo, D.~Th\'eret, and L.~Traynor, see e.g.~\cite{BourgeoisSabloffTraynor15_GeneratingFunctions,EliashbergGromov_GeneratingFamilies,Theret99} and references therein.

From the viewpoint of Cerf theory, it would be enlightening to rigorously name a reasonably behaved category or space of all the 1-parametric families of Morse functions with a given Cerf diagram, and understand how these categories or spaces relate as the Cerf diagram varies, e.g.~ under front homotopies, clasp moves and surgeries, cf.~\cite[Section 4]{Cerf_outline} and \cite[Chapter I.2]{HatcherWagoner}. In the current scientific context, given a Legendrian submanifold $\La\sse (J^1(B),\xi_{\on{st}})$, the following steps would seem both reasonable and an exciting valuable addition to the literature:

\begin{enumerate}
    \item Rigourously set up a stable $\infty$-category $\G\F_\La$ whose objects are given by generating family spectra for $\La$, i.e.~the sublevel set spectra associated to parametric families $f_x:M\lr\R$, $x\in B$, whose Cerf diagram yields the front projection of $\La$ under the natural projection $J^1(B)\lr B\times\R$ sending $j^1f$ to $(b,f(b))$, cf.~\cite[Definition 3.1]{tanaka2025stablehomotopyinvariantlegendrians}. The morphisms should naturally be a spectral enhancement of generating family homology, cf.~e.g.~\cite[Section 1.4]{tanaka2025stablehomotopyinvariantlegendrians}, discussions therein, and variations thereof, as $\G\F_\La$ is preferably constructed to be symmetric monoidal and unital.\\

    \item Show that $\G\F_\La$ is a Legendrian invariant and it is smooth over the sphere spectrum $\mathbb{S}$. In addition, show that the functor that locally records the spectral local system on $\La$ given by the framing of the unstable submanifolds admits a relative left Calabi-Yau structure, as defined in \cite[Section 4.1.10]{brav2023cyclicdeligneconjecturecalabiyau}. This latter property should be implied by the unlinked copy argument use to prove for Sabloff duality, cf.~\cite[Section 4]{MR2284060} or \cite[Prop.~4.1]{EkholmEtnyreSabloff09}, with the relative fundamental class now belonging to the corresponding (relative $S^1$-invariant) topological Hochschild homology.\\
    
    \noindent In the same vein, show the existence of the appropriate functors between $\G\F_{\La}$ and $\G\F_{\La'}$ if $\La$ and $\La'$ are related by Legendrian surgeries or clasp moves, and establish their properties, e.g.~forms of full faithfulness.\\

    \item For a space $\mathsf{g}\mathsf{f}_\La$, rather than a category $\G\F_\La$, I suggest we consider the spectral moduli stack of pseudoperfect objects associated to a stable $\infty$-category, cf.~ \cite[Section 5.3]{AntieauGepner} which extends \cite[Theorem 0.2]{ToenVaquie07} to this setting by taking $R=\mathbb{S}$. If non-empty, one might expect that these spaces $\mathsf{g}\mathsf{f}_\La$ admit rich geometric structures coming from the symplectic topology of Lagrangian fillings of $\La$. For instance, these moduli stacks should be spectrally smooth and the $\infty$-analogue of shifted symplectic, by the categorical properties of (2), and $\mathsf{g}\mathsf{f}_\La$ be equivalent to a substack of $\mathsf{g}\mathsf{f}_{\La'}$ if $\La'$ is a Legendrian surgery of $\La$.\\ 
    
    \noindent Each of the (classical) points of $\mathsf{g}\mathsf{f}_\La$ should be realizable by an extension to the symplectization of the generating spectrum of the corresponding object in $\G\F_\La$, itself generating a possibly immersed (unobstructed) Lagrangian filling of $\La$. By construction, cf.~\cite[Prop.~5.9]{AntieauGepner}, the cotangent complex at such a point shall coincide with the desuspended generating family spectrum of such an extension, crystallizing the fact that the cohomology groups of a Lagrangian filling (or its framed cobordism class or the appropriate spectrum) appearing in the corresponding Seidel isomorphisms are to be understood as infinitesimal deformation groups.\footnote{This is in line with the fact that the cochain complex $C^*(L,k)[1]$ is equivalent to the cotangent complex at the trivial local system in the moduli of pseudoperfect objects for the category of $\infty$-local systems on $L$.}\\
    
    \noindent In addition, these moduli stacks $\mathsf{g}\mathsf{f}_\La$ should admit a generalization of the notion of a cluster structure, by the same symplectic geometric principles as in \cite{CasalsWeng,CGGLSS}, e.g.~an open cluster substack for each submaximal ruling of $\La$. In particular, the braid varieties $X(w_0\beta)$ discussed in the introduction yield, when quotiented by a certain torus action, a stack isomorphic to the simplest connected component of $\mathsf{g}\mathsf{f}_\La$ if $\La\sse(\R^3,\xi_\st)$ is the rainbow closure of $\beta$.\\
\end{enumerate}

\begin{remark}
(i) As per usual in the trichotomy given by Floer theory, generating families and microlocal sheaves, one is to expect an equivalence of stable $\infty$-categories between such spectrally enhanced $\G\F_\La$, perfect modules over the appropriate spectral enhancement of the Legendrian dg-algebra (see ongoing work of Lipshitz-Ng-Sarkar for the case of knots), and with the corresponding category of spectral sheaves on $B\times\R$ with singular support on $\La$.

\noindent (ii) Despite the expected equivalences in (i), there is significant merit in developing and establishing (1),(2) and (3) above within the methods and context of parametric Morse theory and generating families. These different approaches each have their own merits and once each of the facets of this trichotomy is properly developed on its own, they can be better compared and complement each other. It is not ideal to sit on a cuttie-stool if one leg wobbles: each is equally important and can hopefully be appreciated as such.
\qed
\end{remark}

%% file: 2_Examples.tex
\section{A few examples}\label{sec:examples}

Let $\beta=\s_{i_1}\ldots\s_{i_l}$ be an $n$-stranded positive braid word. In order to produce arbitrary examples with ease, the reader is referred to the file ``BruhatCluster.nb'', available in the author's website. In this Mathematica file, I implemented a function
$$\on{BruhatAndClusterVariables}[\{i_1,\ldots,i_l\}, n]$$
that inputs the list of indices $\{i_1,\ldots,i_l\}$ for the crossings of $\beta$ and the number of strands $n$. For each $k\in[1,l-1]$, this function outputs:

\begin{enumerate}
    \item {\it Morse data}: the differential $\dd_k$ between the $k$th and $(k+1)$th crossing, its associated rook matrix in the top-dimensional cell for its maximal ruling, and the corresponding signed product of Bruhat numbers as displayed in the right hand side of \cref{eq:mainidentity}.\\

    \item {\it Cluster data}: the braid matrix associated to the braid consisting of the first $k$ crossings, and the cluster variable associated to the $k$th crossing, as displayed in the left hand side of \cref{eq:mainidentity}.
\end{enumerate}

\noindent In particular, the reader can use it to quickly verify the equality \maintheoref in any reasonably sized example.

\subsection{2-stranded case}\label{ssec:examples_2n_torus} Consider the 2-stranded braid word $\b=\s_1^l\in\Br_2^+$. In this $2\times2$ case, a matrix in $\SL_2$ explicitly factorizes as
\begin{equation}\label{eq:Bruhat2strands1}
\left(\begin{array}{cc}
 a & b \\
 c & d \\
\end{array}
\right)=\left(\begin{array}{cc}
 1 & \frac{a}{c} \\
 0 & 1 \\
\end{array}
\right)\left(\begin{array}{cc}
 0 & b-\frac{ad}{c} \\
 c & 1 \\
\end{array}
\right)\left(\begin{array}{cc}
 1 & \frac{d}{c} \\
 0 & 1 \\
\end{array}
\right)=\left(\begin{array}{cc}
 1 & \frac{a}{c} \\
 0 & 1 \\
\end{array}
\right)\left(\begin{array}{cc}
 0 & -\frac{1}{c} \\
 c & 1 \\
\end{array}
\right)\left(\begin{array}{cc}
 1 & \frac{d}{c} \\
 0 & 1 \\
\end{array}
\right)
\end{equation}
if $c\neq0$, and otherwise factorizes as
\begin{equation}\label{eq:Bruhat2strands2}
\left(\begin{array}{cc}
 a & b \\
 0 & d \\
\end{array}
\right)=\left(\begin{array}{cc}
 a & 0 \\
 0 & d \\
\end{array}
\right)\left(\begin{array}{cc}
 1 & \frac{b}{a} \\
 0 & 1 \\
\end{array}
\right)=\left(\begin{array}{cc}
 a & 0 \\
 0 & \frac{1}{a} \\
\end{array}
\right)\left(\begin{array}{cc}
 1 & \frac{b}{a} \\
 0 & 1 \\
\end{array}
\right).
\end{equation}
Since $n=2$, there is only one type of crossing $\s_{i_j}$, with index $i_j=i_j^*=1$. Therefore, in the above notation for the entries, Equations (\ref{eq:Bruhat2strands1}) and (\ref{eq:Bruhat2strands2}) imply that the Bruhat numbers are $c$ and $1/c$ if $c\neq0$, or $a$ and $1/a$ otherwise. In particular, in the open top-dimensional chart associated to a maximal ruling, where the Bruhat numbers are non-vanishing, the only non-trivial product of Bruhat numbers is the entry $c$ itself. In short, the right hand side of \cref{eq:mainidentity} is simply $\textcolor{blue}{\bbbeta_1(f_k)}$, which can be computed as the lower-left entry of $\dd_k=P_1(z_k)\cdots P_1(z_1)w_0$, $k\in[1,l]$. This entry coincides with the upper-left entry of the braid matrix $B_{\s^k}(z_1,\ldots,z_k)=B_1(z_1)\cdots B_1(z_k)$: such upper-left entry is $\textcolor{purple}{A_k(z_1,\ldots,z_k)}$ by \cite[Section 3.7]{CGGLSS}, and thus \cref{eq:mainidentity} is verified.

\begin{remark}
In this 2-stranded case, the resulting functions in the $z$-variables are related to Euler continuants, e.g.~cf.~\cite[Section 2.2.5]{Hughes_2023} or \cite[Section 5]{MR3837919}.\qed
\end{remark}

\subsection{A 3-stranded example} Consider the 3-stranded braid word $\beta=\s_1\s_2\s_1\s_2\s_1^2\s_2^2\s_1$. Let us consider the Morse function $f_x$ with $x$-coordinate right after the $x$-coordinate of the seventh crossing of $\beta$. This crossing is $\s_2$ and thus $i_7=2$ and $i_7^*=n-i_7=1$ as $n=3$. By the discussion in \cref{ssec:coordinatizing_Morse}, the associated Morse differential can be written as
$$\dd_7=P_2(z_7)P_1(z_6)P_1(z_5)P_2(z_4)P_1(z_3)P_2(z_2)P_1(z_1)w_0:\Z^3\lr\Z^3$$
which reads as
$$\dd_7=\left(
\begin{array}{ccc}
 z_2 z_5-1 & z_1 z_5 & z_5 \\
 z_2 z_4-z_3 & z_1 z_4-1 & z_4 \\
 z_2-\left(z_2 z_5-1\right) z_6+\left(z_2 z_4-z_3\right) z_7 & -z_5 z_6 z_1+z_1-\left(1-z_1 z_4\right) z_7 & -z_5 z_6+z_4 z_7+1 \\
\end{array}
\right).$$
In the locus where the polynomial $\textcolor{orange}{\delta:=-z_5 z_6 z_2+z_4 z_7 z_2+z_2+z_6-z_3 z_7}$ is non-zero, $\dd_7$ we can factorized as $\dd_7=U_1RU_2$ where

$$U_1:=\delta^{-1}\left(
\begin{array}{ccc}
 1 & \left(z_2 z_5-1\right) z_7+z_1 \left(z_4 z_7-z_3 z_5 z_7+1\right) & z_2 z_5-1 \\
 0 & 1 & z_3-z_2 z_4 \\
 0 & 0 & 1 \\
\end{array}
\right),$$
$$R:=\delta^{-1}\left(
\begin{array}{ccc}
 0 & 0 & 1 \\
 0 & z_5 z_6 z_2-z_2+z_1 z_3+z_1 z_4 z_6-z_1 z_3 z_5 z_6-z_6 & 0 \\
 \delta^2 & 0 & 0 \\
\end{array}
\right),$$
$$U_2:=\delta^{-1}\left(
\begin{array}{ccc}
 1 & z_7+z_1 \left(z_5 z_6-z_4 z_7-1\right) & -z_5 z_6+z_4 z_7+1 \\
 0 & 1 & z_4 z_6+z_3 \left(1-z_5 z_6\right) \\
 0 & 0 & 1 \\
\end{array}
\right).$$

\noindent The right hand side of \cref{eq:mainidentity}, which in this case is simply the first Bruhat number $\bbbeta_1(f_x)$ as $i_7^*=1$, is the lower left entry of the rook matrix $R$: it thus equals $\textcolor{blue}{\bbbeta_1(f_x)=\delta}$. The left hand side of \cref{eq:mainidentity}, which is the cluster variable $A_7(z_1,\ldots,z_7)$, is the second leading principal minor of the braid matrix
$B:=B_{\s_1\s_2\s_1\s_2\s_1^2\s_2}(z_1,\ldots,z_7)$, as $i_7=2$. This braid matrix which reads
$$B=\tiny\left(
\begin{array}{ccc}
 z_2-z_1 z_3+\left(-z_2 z_5+z_1 \left(z_3 z_5-z_4\right)+1\right) z_6 & z_1+\left(z_2 z_5+z_1 \left(z_4-z_3 z_5\right)-1\right) z_7 & -z_2 z_5+z_1 \left(z_3 z_5-z_4\right)+1 \\
 z_3 \left(z_5 z_6-1\right)-z_4 z_6 & \left(z_4-z_3 z_5\right) z_7+1 & z_3 z_5-z_4 \\
 z_5 z_6-1 & -z_5 z_7 & z_5 \\
\end{array}
\right).$$
\noindent Therefore $\textcolor{purple}{A_7(z_1,\ldots,z_7)=\delta}$ and \cref{eq:mainidentity} holds.

\subsection{A 4-stranded example} Consider the 3-stranded braid word $\beta=\s_1\s_2\s_3\s_1\s_3\s_2\s_1\s_3\s_2^2\s_1$. Let us consider the Morse function $f_x$ with $x$-coordinate right after the $x$-coordinate of the seventh crossing of $\beta$. This crossing is a $\s_1$ and thus $i_7=1$ and $i_7^*=n-i_7=3$ as $n=4$. By \cref{ssec:coordinatizing_Morse}, the associated Morse differential is
$$\dd_7=P_1(z_7)P_2(z_6)P_3(z_5)P_1(z_4)P_3(z_3)P_2(z_2)P_1(z_1)w_0:\Z^4\lr\Z^4$$
which reads
$$\dd_7=\left(
\begin{array}{cccc}
 -z_3 & -z_2 & -z_1 & -1 \\
 -z_3 z_7 & -z_2 z_7-1 & -z_1 z_7 & -z_7 \\
 -z_3 z_6 & -z_4-z_2 z_6 & -z_1 z_6-1 & -z_6 \\
 1-z_3 z_5 & -z_2 z_5 & -z_1 z_5 & -z_5 \\
\end{array}
\right).$$

\noindent In the locus where
$$\delta_1:= 1-z_3 z_5,\quad\delta_2:=z_4 \left(1-z_3 z_5\right)+z_2 z_6,\quad\textcolor{orange}{\delta_3:=z_3 z_5-z_2 z_7+z_1 \left(z_4 z_7-z_6\right)-1},$$ are non-zero, $\dd_7$ we can factorized as $\dd_7=U_1RU_2$ where

$$R:=\left(
\begin{array}{cccc}
 0 & 0 & 0 & \delta_3^{-1} \\
 0 & 0 & \frac{-z_3 z_5+z_2 z_7+z_1 z_6-z_1 z_4 z_7+1}{z_4 \left(1-z_3 z_5\right)+z_2 z_6} & 0 \\
 0 & \frac{-z_3 z_5 z_4+z_4+z_2 z_6}{z_3 z_5-1} & 0 & 0 \\
 1-z_3 z_5 & 0 & 0 & 0 \\
\end{array}
\right),$$
\noindent and see the code above for $U_1$ and $U_2$. Therefore, given that $i_7^*=3$ and the rook matrix $R$ above, the right hand side of \cref{eq:mainidentity} in this case is
$$\textcolor{blue}{\bbbeta_1(f_x)\bbbeta_2(f_x)\bbbeta_3(f_x)=\delta_1\cdot\frac{\delta_2}{\delta_1}\cdot\frac{\delta_3}{\delta_2}=\delta_3}.$$
\noindent i.e.~the product of the first three entries of $R$ from the lower-left corner. The left hand side of \cref{eq:mainidentity}, i.e.~the cluster variable $A_7(z_1,\ldots,z_7)$, is the first leading principal minor of the braid matrix
$B:=B_{\s_1\s_2\s_3\s_1\s_3\s_2\s_1}(z_1,\ldots,z_7)$, since $i_7=1$. This braid matrix is

$$B=\left(
\begin{array}{cccc}
 z_3 z_5-z_2 z_7+z_1 \left(z_4 z_7-z_6\right)-1 & z_2-z_1 z_4 & z_1 & -z_3 \\
 z_4 z_7-z_6 & -z_4 & 1 & 0 \\
 z_7 & -1 & 0 & 0 \\
 z_5 & 0 & 0 & -1 \\
\end{array}
\right).$$
\noindent Thus $\textcolor{purple}{A_7(z_1,\ldots,z_7)=\delta_3}$ as well, illustrating \cref{eq:mainidentity}. 